\documentclass{article}
\usepackage{amsmath,amssymb,latexsym,amsthm,color}

\newcommand{\KK}{\mathbb{K}}
\newcommand{\QQ}{\mathbb{Q}}
\newcommand{\RR}{\mathbb{R}}
\newcommand{\NN}{\mathbb{N}}

\newcommand{\LL}{\mathbb{L}}

\newcommand{\coker}{\mathrm{coker}\kern.5pt}

\theoremstyle{plain}
\newtheorem{theorem}{Theorem}[section]
\newtheorem{corollary}[theorem]{Corollary}
\newtheorem{proposition}[theorem]{Proposition}
\newtheorem{lemma}[theorem]{Lemma}

\theoremstyle{definition}

\newtheorem{example}[theorem]{Example}

\newtheorem{remark}[theorem]{Remark}

\def \e {\epsilon}

\begin{document}

\title{ The intermediate value theorem over a non-Archimedean field via Hensel's Lemma}

\author{L. Corgnier, C. Massaza, P. Valabrega}
 
\date{}

\maketitle

\begin{abstract} The paper proves that all power series over a maximal ordered Cauchy complete non-Archimedean field satisfy the intermediate value theorem on every closed interval. Hensel's Lemma for restricted power series is the main tool of the proof.
\end{abstract}


\textbf{MSC\,2010:} 12J15.

\section{Introduction}
It is well-known that, over a complete Archimedean field (i.e. the field $\RR$ of real numbers), the intermediate value theorem holds for every function continuous on a closed interval. The proof is strongly based on the Archimedean property, or better on the dichotomy procedure, which is  a consequence of it.

Over a non-Archimedean field $\KK$ the theorem is false for the class of all continuous functions (see \cite{Corgnier}, Example 4.1), even in the event that $\KK$ is maximal ordered and Cauchy complete. Nevertheless long ago  it has been proved for polynomials and rational functions  defined over a maximal ordered, not necessarily complete, non-Archimedean field (see \cite{Bourbaki A}, \textsection 2, n. 5, Proposition 5).

In \cite{Corgnier} we proved that, when $\KK$ is maximal ordered, the theorem can be extended to any power series algebraic over the field $\KK(X)$ of rational functions, with a direct proof that makes use of  the theorem for polynomials. 

In the present paper we investigate the intermediate value theorem for an arbitrary power series over a non-Archimedean maximal ordered and complete field $\KK$. In order  to attain the result we use Hensel's Lemma, in its strong version for the ring of restricted power series.  In fact we apply it to the valuation ring of all elements of $\KK$ that are not infinitely large with respect to $\QQ$ (or to any other subfield $\LL$ which $\KK$ is non-comparable with). Since the maximal ideal of infinitesimals may contain no topologically nilpotent element, we need the stronger version of Hensel's Lemma proved in \cite{Valabrega}.
 
We  also show  that each root of odd order of a power series is the  limit of a sequence of roots of the partial sums  (the order, or multiplicity, being defined in \cite{Corgnier}, Theorem 3.11). The even order behaves quite differently; in fact it may occur, and sometimes it occurs, that the root of the series is only the limit of a sequence of extremes of the partial sums.

It is worth observing that, while in \cite{Corgnier} we proved that {\it{algebraic}} power series satisfy the intermediate value theorem on a closed interval of a {\it{maximal ordered}} field, here we prove that {\it{general}} power series have their zeros in a {\it{maximal ordered and complete}} field.

As for the cardinality of the set of zeros, we show that, on the whole field $\KK$, it may occur that there are infinitely many roots, but they shrink to finitely many on the subring of the  elements that are not infinitely large with respect to a subfield which $\KK$ is non-comparable with.

We want to recall that the theory of non-Archimedean ordered fields goes back to the nineteenth century and was introduced by Veronese and Levi-Civita (see \cite{Levi-CivitaV}, \cite{Levi-CivitaL}, and also \cite{Berz} and \cite{Shamseddine-Berz2}). As for the intermediate value theorem  on the  Levi-Civita field of functions from $\QQ$ to $\RR$ with left-finite support, a proof is given in \cite{Shamseddine-Berz}, while  in \cite{Laugwitz} and in \cite{Robinson} a quite different point of view is considered.

\section{Notation and general facts}
Unless otherwise stated, $\KK$ is a {\it{non-Archimedean, maximal ordered, complete field}} (for ordered fields and completions in general we refer the reader to \cite{Bourbaki A}, \cite{Lang} and \cite{Waerden}).  We only recall that $\KK$ is called {\it{maximal ordered}} if every ordered algebraic extension of $\KK$ coincides with $\KK$ (\cite{Bourbaki A}, \textsection 2, n. 5, D\' efinition 4) and that $\KK$ is maximal ordered if and only if every positive element of $\KK$ is a square and moreover every odd degree polynomial  over $\KK$ has a root in $\KK$ (\cite{Bourbaki A}, \textsection 2, n. 6, Th\' eor\`eme 3).  

We assume once for all that the order topology has a countable basis for the neighbourhoods of $0$ (see \cite{Kelley}, ch. I, p. 50 and \cite{Massaza}, X, p. 335). An element $x \in \KK$ is {\it{topologically nilpotent}} if $\lim x^n = 0$ (see \cite{Greco}, p. 19).

The following cases can occur (see \cite{Corgnier}, \textsection 3):

\begin{enumerate}

\item there is a topologically nilpotent element $\epsilon \in \KK$,

\item  there is a sequence $(\e_0 > \epsilon_1 > \epsilon_2 > \cdot \cdot \cdot )$, converging to $0$,  such that $\forall n,\epsilon_n > 0$ and $\forall n, \forall i, \epsilon_n^i > \epsilon_{n+1}$. We always  choose $\e_0 = 1$. 

\end{enumerate} 

In case 1. the sets $U_n = \{x \in \KK, |x| < \e^n\}$ give a basis for the neighbourhoods of $0$, while in case 2. we need the sets $\bar U_n = \{x \in \KK, |x| < \e_n\}$.

For every non-Archimedean ordered field $\LL, \bar \LL$ denotes the ordered closure and $\hat \LL$ the Cauchy completion.

If $S(X) = \sum a_nX^n$ is a power series over $\KK$, we denote by $S_n(X)$ the $n$-th partial sum $\sum_{i = 0}^n a_iX^i$.

If $\LL \subset \KK$ is a  subfield such that  $\KK$ contains at least one element $x$ larger than every $a \in \LL$, we say that $\KK$ is non-comparable with $\LL$  (or non-Archimedean over $\LL$). For instance, $\KK$ is non-comparable with $\QQ$. 

An element $x \in \KK$ is called infinitely large with respect to $\LL$ if $|x| > a, \forall a \in \LL^+$ (set of all positive elements). An element $y \in \KK, y \ne 0$ is called infinitely small (or infinitesimal) with respect to $\LL$ if $|y| < a, \forall a \in \LL^+$. If $x$ is infinitely large, then $\frac{1}{x}$ is infinitely small and conversely.
When {\it{$\LL = \QQ$ we simply say that $x$ is infinitely large and $y$ is infinitely small or infinitesimal. }}
If $x \in \KK$ is algebraic over the subfield $\LL$, then $x$ is neither infinitely large nor infinitely small with respect to $\LL$ (see \cite{Bourbaki A}, Exercise 14, p. 57). The same is true for every $y \in \hat L$ (as a consequence of the definition of order in the completion, see \cite{Waerden}, 67). 

Given $\LL \subset \KK$ such that $\KK$ is non-comparable with it, we set: 

$A_{\LL} = \{x \in \KK, x\ \text{is not infinitely large with respect to}\ \LL\}$, 

$M_{\LL} = \{x \in \KK, x\ \text{is infinitely small with respect to}\ \LL\}.$

Then $A_{\LL}$ is a subring of $\KK$ (\cite{Bourbaki A}, Exercise 1 p. 53) and  $M_{\LL}$ is a maximal ideal of $A_{\LL}$ (\cite{Bourbaki A}, Exercise 11, b), p. 57). Moreover  $A_{\LL}$ is a valuation ring, since either $x $ or $\frac{1}{x}$ belongs to $A_{\LL}$ (\cite{BourbakiEV},  chap. VI, n. 2, Th\' eor\`eme 1).

\begin{remark}  Let us assume that there is  a sequence $(\epsilon_n)$ converging to $0$ but not a single $\epsilon$ such that $\lim \epsilon^n = 0$. 

(a) Given a subfield $\LL$ over which $\KK$ is non-Archimedean, we can assume that all the $\epsilon_n$'s (except $\e_0 = 1$) are infinitesimal  with respect to  $\LL$ (by possibly discarding finitely many of them). 

It follows that every $\omega_n = \frac{1}{\e_n}, n \ge 1,$ is infinitely large with respect to $\LL$.

(b) It holds: 
 $ \epsilon_1 \epsilon_n > \epsilon_{n+1}$, since $\epsilon_1 > \epsilon_n$ implies $\epsilon_1 \epsilon_n > \epsilon_n^2 > \epsilon_{n+1}$ (this inequality actually follows from 2. above).

(c) It is easy to see that each $\e_{i+1}$ is transcendental over $\LL(\e_1,\e_2,\cdot \cdot \cdot,\e_i) \subset \KK$, since algebraic elements cannot be infinitesimal.

\end{remark}

\begin{remark}If there is a topologically nilpotent $\epsilon$, then $\e$  does not belong to every $\LL$  which $\KK$ is non-comparable with. In fact,  $\e \in \LL$ impliest $\e^n \in \LL, \forall n$, in contradiction with the non-Archimedean assumption on $\KK$ (choose $x \in \KK$  infinitely small with respect to $\LL$, so less than every $\epsilon^n$:  $x$ must be $0$).
\end{remark}

A local ring $(A,M)$,  is called henselian if the following property holds:

Let $P(X) \in A[X]$ be a polynomial such that its canonical image $\bar {P}(X)$ into the quotient  ring $\frac{A}{M}[X]$ is the product $\bar{Q}(X) \bar{T}(X)$ of a monic polynomial $\bar{Q}(X)$ and another polynomial $\bar{T}(X)$, the two factors being coprime. Then $P(X) = Q(X)T(X)$, where $Q(X)$ is a monic polynomial that lifts $\bar{Q}(X)$ and $T(X)$ is a polynomial that lifts $\bar{T}(X)$. Moreover $P(X),Q(X)$ are uniquely determined and coprime.

Hensel's Lemma states that a complete local ring is henselian (see \cite{Nagata}, ch. V, \textsection 30, p. 103). There is a wide class of henselian rings (see for instance \cite{Greco2}).

\bigskip
A monic polynomial $X^r+\cdot \cdot \cdot +c_1X+c_0$ is an {\it{N-polynomial}} if $c_0 \in M, c_1 \notin M$. A local ring is {\it{henselian}} if and only if every N-polynomial has a root in $M$ (\cite{Nagata}, ch. V, \textsection 43, p. 179 and also \cite{Greco2}, Theorem 5.11).

If $S(X)$ is defined over a topological ring $A$ with a linear topology, i.e. with a basis of the neighbourhoods of $0$ formed with ideals, $S(X)$ is called {\it{restricted }} if $\lim a_n = 0$ (see \cite{Greco}, p. 18).

Hensel's Lemma can also be given for restricted power series (see \cite{Bourbaki AC}, III, \textsection 4, n. 3 and \cite{Greco}, Theorem 3.7, p. 19).

\begin{lemma} \label {Hensel} (Hensel's Lemma for restricted power series) Let $A$ be a complete separated ring with respect to a linear topology, $M$ a closed ideal whose elements are topologically nilpotent, $S(X)$ a restricted power series such that its canonical image $\bar {S}(X)$ into the topological quotient  ring $\frac{A}{M}\{X\}$ is the product $\bar{P}(X) \bar{T}(X)$ of a monic polynomial $\bar{P}(X)$ and a restricted series $\bar{T}(X)$, the two factors being coprime. Then $S(X) = P(X)T(X)$, where $P(X)$ is a monic polynomial that lifts $\bar{P}(X)$ and $T(X)$ is a restricted series that lifts $\bar{T}(X)$. Moreover $P(X),T(X)$ are uniquely determined and coprime.
\end{lemma}

 \section{Non-comparable subfields and some topology }

In this section we investigate some topological properties of the valuation ring of non-infinitely large elements that follow from the ordering of the field $\KK$, with special focus on topologically nilpotent elements.

\begin{lemma} \label{nilpotent} If there is a topologically nilpotent element $\epsilon$, then $X$ contains a maximal element, which is both maximal ordered and Cauchy complete.
\end{lemma}

\begin{proof}

The set $X$ of all subfields $\LL \subset \KK$ which $\KK$ is non-comparable with is not empty and can be ordered by inclusion.

First of all we show that,  if $F$ is a subset of $X$, then $\LL' = \cup F$  belongs to $X$.  

In order to prove this property let us recall (Remark 2.2 above) that no $\LL \in X$ can contain $\epsilon$.  
As a consequence $\epsilon \notin \LL'$,  so that  $\KK$ is non-comparable with $\LL'$.

Therefore, by Zorn's lemma, $X$ has a maximal element $\LL$.   Such a subfield is both maximal ordered and complete, since the ordered closure and the completion of an ordered field are comparable with it.
  
 \end{proof}
 
 \begin{proposition} \label{nilpotent2} If $\KK$ contains a topologically nilpotent element $\e$ and $\LL$ is maximal as above, then every element of $M_{\LL}$ is topologically nilpotent.
 \end{proposition}
 
 \begin{proof} Let $x$ be in $M_{\LL}$. Put: $y = \frac{1}{x}$ and observe that $|y| >  |h|, \forall h \in \LL$, so that $y \notin \LL$, which implies that $y$ is transcendental over $\LL$ (every algebraic element is comparable with $\LL$). Let us assume $y > 0$,  and  consider the field $\LL(y) \subset \KK$. 
Since $\LL$ is maximal with the property that $\KK$ is non-comparable with it,  $\KK$ is comparable with $\LL(y)$, i.e. every $k \in \KK$ is neither infinitely large nor infinitely small with respect to  $\LL(y)$, so that there is a rational function $\frac{P(y)}{Q(y)}$ which is larger than $\omega = \frac{1}{\e}$. But, in the unique ordering of $\LL(y)$ (see \cite{Massaza4}, \textsection 3) every rational function is less than some power $y^m$, so that $\omega < y^m$ for some $m \ge 1$.  Therefore $\lim x^n = \lim \frac{1}{y^m} = 0$.
  \end{proof}
 
 \begin{remark}The above two results state that, if $\KK$ contains a topologically nilpotent element $\epsilon$, then, among all those subfields which $\KK$ is non-comparable with, there is a maximal element $\LL$, giving rise to a local ring $(A_{\LL},M_{\LL})$ whose elements in $M_{\LL}$ are all topologically nilpotent.  
 \end{remark}
  \begin{remark} \label {topology} Let $\LL$ be any subfield with which $\KK$ is non-comparable.  In all cases, with or without topologically nilpotent elements, the ring $A_{\LL}$, equipped with the topology induced by the ordering of $\KK$, is a {\it{topological ring}}, i.e. the functions $(x,y) \to (x+y), (x,y) \to xy$ and $x \to -x$ are continuous. Moreover  it is {\it{Hausdorf}} since, if $|x| \le \e_n (\e^n), \forall n$, then $x = 0$ (see \cite{Greco}, p. 2).  
 
 Observe that $\KK$ is a topological field and every subring is a topological ring.
 \end{remark}

\begin{lemma} \label{linear topology} For every subfield $\LL$ such that  $\KK$ is non-comparable with it, the topology of $A_{\LL}$ is linear, i.e. there is a basis of the neighbourhoods of $0$ whose elements are ideals of $A_{\LL}$.
\end{lemma}

\begin{proof} Case 1: there is a sequence of infinitesimal elements  $(\e_n, n\in \NN)$ and no nilpotent element.  Set $U_n = \{x \in A_{\LL}, |x| < \e_n\}, V_n= \epsilon_nA_{\LL}$.  It is enough to show  that $U_{n+1} \subset V_{n+1} \subset U_n, \forall n$.

Step 1 ($U_{n+1} \subset V_{n+1}$). Assume that $x \in U_{n+1}$; this means that $a = \frac{x}{\epsilon_{n+1}} \in A_{\LL}$, so that $x \in V_{n+1}$.

Step 2 ($V_{n+1} \subset U_n$). Let $x$ be in $V_{n+1}$, i.e $x = a\epsilon_{n+1}$, for some $a \in A_{\LL}$ and set $\omega_1 = \frac{1}{\epsilon_1}$. Since $|a| < \omega_1, \forall a \in A_{\LL}$ (see Remark 2.1 (a)), it follows that $|x| < \epsilon_{n+1}\omega_1 < \epsilon_n$ by our choice of the infinitesimals, which proves the claim.   

Case 2: there is a topologically nilpotent element $\e$. Then the proof above works if we replace $\e_n$ by $ \e^n, \forall n \in \NN$ (and use remark 2.2).
\end{proof}

\begin{proposition} \label{complete} For every subfield $\LL$ such that $\KK$ is non-comparable with it the following hold true:

1. $A_{\LL}$ is closed in $\KK$ and Cauchy complete 

2.  $M_{\LL}$ is  closed in $A_{\LL}$.

\end{proposition}

\begin{proof} 

1. It is enough to prove that $A_{\LL}$ is complete.

Let $(c_n)$ be any Cauchy sequence with $c_n \in A_{\LL}, \forall n$.  Hence there is $c \in \KK$ such that $c = \lim c_n$.  This implies that,   for every infinitesimal $h$, the open interval $]c-h, c+h[$ contains at least one element $c_N$, i.e. $c-c_N = \eta$, with $\eta$ infinitely small and so belonging to $M \subset A$. We obtain that $c-c_N = \eta \in A_{\LL}$, i.e. that $c = c_N+\eta \in A_{\LL}$.

2. Choose $d \in A_{\LL}$ such that every open neighbourhood of  infinitesimal radius $h$ of $d \in A_{\LL}$ contains some $x \in M_{\LL}, x \ne d$. This means that $ |x-d| < h$, i.e. $x-d \in M_{\LL}$, so that $d \in M_{\LL}$.
\end{proof}

We will consider the special case  $\LL= \QQ $, since we know that $\QQ$ is a  subfield which $\KK$ is non-comparable with.  The ring $A = \{x \in \KK,  x$ is not infinitely large over $\QQ\}$ is a valuation ring, as we have already seen, and $M = \{x \in A, x$ is infinitely small over $\QQ\}$ is its maximal ideal.

 \section{Changing the interval and trasforming $S(X)$ into a restricted power series over $A$}

We want to show that every power series converging in some closed interval, can be transformed into a restricted power series over the ring of non-infinitely large elements,   the simple tools being a linear change of variable and the multiplication by a suitable element in $\KK$. The series so acquires a few good properties useful in what follows.

\begin{proposition} \label{restricted} Let $S(X)$ be a power series defined over a set $D_S \subset \KK$ containing the closed interval $[a,b]$. Then there is a linear change of variable $X = hZ+k$ such that 

(i) $[a,b]$ is mapped one-to one onto $[1,2]$, and $X = a$ corresponds to $Z = 1$, $X = b$ corresponds to $Z = 2$,

(ii) $S(hZ+k) = T(Z) $ is a power series whose domain contains $[1,2]$.

Moreover there are  $d \in \KK$ and $N \in \NN$ such that 

(iii) $dT(Z)$ is a restricted power series over the ring $A$ of non-infinitely large elements and  $da_N = 1, da_{N+h} \in M, \forall h \ge 1$.

\end{proposition} 

\begin{proof}
(i) and (ii).

We set:
$ X =(b-a)Y$, obtaining a series in the variable  $Y$, say $U(Y) = \sum a_n(b-a)^nY^n$. $U(Y)$ is convergent at least on $ [\frac{a}{b-a}, \frac{b}{b-a}]$.
Now we set: $ \bar k=\frac{2a-b}{b-a}$ and operate the translation $Y = Z+\bar k$. This translation is allowed if $U(\bar k)$ is a converging series (see \cite{Corgnier}, Theorem 3.7). This holds true since  $U(Y)$ is convergent both at  $\frac{a}{b-a}$ and at $\frac{b}{b-a}$, so also at $\frac{2a}{b-a}$ and at $\frac{2a}{b-a} - \frac{b}{b-a}$ (see \cite{Corgnier}, Theorem 3.3). Hence $T(Z) = S((b-a)Z+(2a-b)) = U(Y)$,  is convergent on $[1,2]$ (see \cite{Corgnier}, Theorem 3.7) and $S(a) = T(1), S(b) = T(2)$, so that (ii) is fulfilled with $h = b-a, k = 2a-b$.

(iii). 

Set: $T(Z) = \sum t_nZ^n$. We have $\lim_{n \to \infty} t_n = 0$ because the series is convergent at $Z = 1$ (see \cite{Corgnier}, General facts, Theorem 2.1, and \cite{Massaza}, p. 335, XII)  and so only finitely many coefficients lie outside of $A$, since $A$ contains all the infinitely small elements.

Let now $|t_h| = a$ be the largest among all the absolute values $|t_n|$. Then $b_n = \frac{t_n}{a}$ is not infinitely large and so $H(X) = \frac{S(X)}{a} = \sum b_nX^n$ is a power series over $A$ such that $\lim b_n = 0$ (it is convergent at $1$). Therefore $H(X) \in A\{X\} = $ ring of restricted power series over $A$. Observe that $|b_h| = 1$ implies $H(X) \notin M\{X\}$ (restricted series with all coefficients in $M$). In this event  there is the largest integer, say $N$, such that $b_N \notin M$. As a consequence $b_N$ is invertible in $A$ and we can consider the following series: $V(X) = \frac{H(X)}{b_N} = \sum c_nX^n$, which is still restricted over $A$ and has $c_N = 1, c_n \in M, \forall n > N$.  Therefore $d = \frac{1}{ab_N}$.

\end{proof}

\begin{remark} \label{remark restricted} 

(i) If $[a,b]$ is any interval, we can transform $S(X)$ into another series $T(Z)$ converging in $[1,2]$ and then we can replace $T(Z)$ by $dT(Z)$. Observe that $T(1)T(2) < 0$ if and only if $S(a)S(b) < 0$.

(ii) It is worth to point out that there is a one-to-one correspondence between the zeros of $S(X)$ in $[a,b]$ and the zeros of $T(Z)$ in $[1,2]$.

 (iii) Obviously the above proof works if {\it{$[1,2]$ is replaced by any interval whose endpoints are neither infinitely large nor infinitely small.}}

\end{remark}

\section{Hensel's lemma for restricted power series}

In the present section $\KK$ is {\it{maximal ordered and complete, with the exception of the following Proposition \ref{Henselpair}, which holds true for a maximal ordered field $\KK$}}. We choose a subfield $\LL \subset \KK$ such that $\KK$ is not comparable with it, for instance $\LL = \QQ$.  We want to show that Hensel's Lemma for restricted power series (\cite{Greco}, p.19) holds on the local ring $A$ of  elements which are non-infinitely large with respect to $\LL$, even if there is no topologically nilpotent element in the maximal ideal $M$. By \cite{Valabrega}, Teorema 5, it is enough to show that $(A,M)$ is a henselian pair.

\begin{proposition} \label{Henselpair} Let $(A,M)$ be a valuation ring of a maximal ordered, not nessarily complete, field $\KK$. Then $(A,M)$ is a henselian pair.
\end{proposition}

\begin{proof}

 Since $A$ is a local ring, it is enough to prove that every $N$-polynomial $P(X) = X^r+c_{r-1}X^{r-1}+\cdot \cdot \cdot+c_1X+c_0 \in A[X]$ (i.e. with $c_0 \in M, c_1 \notin M$) has a root in $M$ (see Notation and  \cite {Valabrega}, \textsection 1). 
 
 First we observe that, if the polynomial has degree $r = 1$, then $P(X) = X+c_0$ has the root $-c_0 \in M$.

Then we consider the case of a degree two polynomial. If $P(X) = X^2+2bX+c$ is any $N$-polynomial, then it has two roots in $\KK(i)$ (algebraic closure of $\KK$, see \cite{Bourbaki A}, \textsection 2, n. 6, Th\' eor\`eme 3): $a = -b+\sqrt{b^2-c}, a' = -b-\sqrt{b^2-c}$. Since $b \notin M, c \in M, b^2-c$ is for sure positive, so that the two roots actually belong to $\KK$ (which is maximal ordered), hence also to $A$, which is integrally closed as a valuation ring (\cite{BourbakiEV}, chap. 6, \textsection 1, n. 3, Corollaire 1).  Therefore there is no irreducible degree two $N$-polynomial. 

Now we point out that, if an N-polynomial $p(X) = X^m+\cdot \cdot \cdot +p_1X+p_0$ is the product of two factors $q(X) = X^s+\cdot \cdot \cdot +q_1X+q_0 , r(X) = X^h+\cdot \cdot \cdot +r_1X+r_0$, then one and only one between the two factors is an N-polynomial. In fact we have: $q_0r_0 = p_0 \in M$, so that one factor must belong to $M$, say $q_0$. In this event $q_0r_1+q_1r_0 = p_1 \notin M$, so that neither $q_1$ nor $r_0$ can belong to $M$. Therefore $q(X)$ is an N-polynomial while $r(X)$ is not. This implies in particular that a degree two $N$-polynomial has exactly one root in $M$.

Let us now assume that $r \ge 3$. Since $\KK$ is maximal ordered, $P(X)$  is the product of linear factors, say $P_1(X),\cdot \cdot \cdot, P_h(X)$, and irreducible second degree factors, say $Q_1(X),\cdot \cdot \cdot, Q_s(X)$ (see (\cite{Bourbaki A}, \textsection 2, n. 6, Proposition 9). Since $A$ is integrally closed, all factors have coefficients in $A$ (see \cite{Bourbaki AC}, chap. 5, \textsection 1, n. 3, Proposition 11). 

Therefore we obtain that one and only one among the linear factors is an N-polynomial, since the second degree irreducible factors cannot be N-polynomials. Such a factor has the required root.

\end{proof}

\begin{corollary} \label{Henselrestricted} $A$ satisfies Hensel's Lemma for restricted power series.
\end{corollary}
\begin{proof} It is \cite{Valabrega}, Teorema 5, since $A$ is complete and Hausdorf with respect to a linear topology, $M$ is closed and $(A,M)$ is a henselian pair.

\end{proof}

\begin{corollary} \label{Salmon}Let $S(X) = \sum_{n = 0}^{\infty} a_nX^n$ be a restricted power series over $A$ such that the partial sum $S_N(X)$ is a monic polynomial for some $N$ and moreover $a_{N+h} \in M, \forall h \ge 1$. Then $S(X) = P(X)B(X)$, where $P(X)$ is a monic polynomial such that $P(X) = S_N(X)$ mod $M$ and $B(X) \in 1+M\{X\}$ is a restricted power series.
\end{corollary}

\begin{proof} This is essentially \cite{Salmon}, Th\' eor\`eme 10. In fact the proof of this theorem makes only use of Hensel's Lemma for restricted power series, applied to the decomposition (mod $M$): $\bar S(X) = \bar S_N(X) \cdot \bar 1$. The previous corollary states that such a lemma holds without topologically nilpotent elements.

 \end{proof}

\begin{corollary}{\label{finitely many}} Let $S(X) = \sum_{n = 0}^{\infty} a_nX^n$ be a power series defined in the closed interval $[a,b]$. Then $S(X)$ has only finitely many zeros. 

\end{corollary}

\begin{proof}
Thanks to Proposition \ref{restricted}, we can assume that the partial sum $S_N(X)$ is a monic polynomial for some $N$, while $a_{N+h} \in M, \forall h \ge 1$.  By Corollary \ref{Salmon} we have $S(X) = P(X)B(X)$, where $P(X)$ is a polynomial and $B(X) \in 1+M\{X\}$ cannot have roots; therefore $S(X)$ vanishes where a polynomial vanishes.
\end{proof}

\begin{remark} In all the above results any field $\LL$ such that $\KK$ is non-comparable with it works. We obtain Corollary \ref {Salmon} with $A = $ ring  of elements that are not infinitely large with respect to $\LL$.
\end{remark}

\section{The intermediate value theorem}
\begin{theorem} \label{intermediate}Let $S(X)$ be a power series over $\KK$ converging at least in $[a,b]$ and such that $S(a)S(b) < 0$. Then there is $c \in ]a,b[$ such that $S(c) = 0$.

\end{theorem}

\begin{proof}
Thanks to Proposition  \ref{restricted} we can assume that

- $a = 1, b = 2$,

- $S(X)$ is a power series over the local ring $(A,M)$, where $A = A_{\LL}$ is the ring of elements that are not infinitely large over any subfield $\LL \subset \KK$, with which $\KK$ is non-comparable, while $M = M_{\LL}$ is the maximal ideal of all the infinitely small elements, 

- $S(X)$ is restricted because it is convergent at $X = 1$, which implies $\lim a_n = 0$,

- $S(X)$ has a coefficient $a_N = 1$ with the property that $a_m \in M, \forall m > N$.

 Therefore, by Corollary \ref{Salmon}, $S(X) = P(X)B(X)$, where $P(X) = $ monic polynomial over $A$ such that $\bar P(X) = \bar S(X), B(X) = \sum b_nX^n =$ restricted power series over $A$ belonging to $1+M\{X\}$. If $x \in [a,b]$, then $B(x) > 0$, since $B(x) = 1+m$, for a suitable infinitesimal $m \in M$.

Let us now assume that $S(a)S(b) < 0$. Since $Q(x) > 0$ everywhere in the interval, we must have $P(a)P(b) < 0$ and so, by the intermediate value theorem for polynomials, there is $c \in ]a,b[$ such that $P(c) = 0$.   This implies $S(c) = 0$.
\end{proof}

\begin{remark} 
In the proof above any $\LL$, such that $\KK$ is non-comparable with it, can work, in particular $\LL = \QQ$.
When $\KK$ contains a topologically nilpotent element $\epsilon$ the proof of the intermediate value theorem can be based on Hensel's Lemma as it is stated in \cite{Greco}, p. 19 provided that we choose a maximal subfield $\LL$ which $\KK$ is non-comparable with.  It is in fact enough to observe that Hensel's Lemma for restricted power series holds true, since $A$ is equipped with a linear topology (see Lemma \ref{linear topology}) and it is Hausdorf and complete as well as the maximal ideal $M$ (Lemma \ref{complete}), while every element of $M$ is  topologically nilpotent (Lemma \ref{nilpotent2}).  Therefore the proof based on the decomposition mod $M$ works.\end{remark}

\section{The set of zeros of a power series: accumulation and cardinality}

The following theorem and corollary  hold both in the Archimedean and in the non-Archimedean case.

\begin{theorem} \label{accumulation}
Let  $ S(X) $ be a power series such that $S(a)S(b) < 0$. Then at least one among the zeros of $S(X)$ in $[a,b]$  is an accumulation point for the set $Z = \cup Z_n$, where $ Z_n = \{ \{z \in [a,b], S_n(z) = 0\}$. Therefore at least one zero is the limit of a sequence of zeros of partial sums.

\end{theorem}

\begin{proof}

Let $( c_1 < c_2 < \cdot \cdot \cdot < c_k)$ be the finitely many zeros of $S(X)$ in the interval. 

Let us assume $S(a) < 0, S(b) > 0$ and that  $c_1$ be not an accumulation point. Then there is an interval $ I = [c_1-r,c_1+r]$ contaning no element of the set $Z$, no zeros of $S(X)$ except $c_1$ and neither $a$ nor $b$. This implies that the partial sum $S_n(X)$ has no root in $I$ and that  $S_n(x)$ is either $ > 0$ or $< 0, \forall x \in I$.
Since $S(x) = \lim_{n \to \infty} S_n(x), \forall x \in I$, there is $n_0$ such that $n > n_0$ implies $S_n(x) < 0, \forall x \in I, x < c_1$; as a consequence $S_n(x) < 0, \forall x \in I$.   
 
We conclude that, if $ c_1$ is not an accumulation point,  then there is a suitable $r > 0$ such that $S(c_1+r)  < 0$. We now shrink $[a,b]$ to 
the subinterval $ [c_1+r,b]$ so that the least  zero of $S(X)$ becomes $c_2$. We repeat the argument and assume that $c_2$ is not an accumulation point.  After finitely many steps we find that, if no zero is an accumulation point, then there is   $s$ such that  $S(c_k-s) < 0, S(c_k+s)  < 0$. But $ S(b) > 0$ implies that there is a zero between $ c_k+s$ and $ b$, which is a contradiction.

If $c$ is an accumulation point of a sequence,  then it is obvious that it is the limit of a suitable subsequence.
\end{proof}
\begin{remark} By \cite{Corgnier}, Theorem 3.11, it is easy to define the concept of order (or multiplicity) of a power series at $c \in \KK$: if $S(X) = (X-c)^sq(X)$, where $q(c) \ne 0, s$ is the order. Therefore we are allowed to use the terms {\it{odd order, even order}}. It is clear that $S(X)$ has odd order at $c$ if and only if $S(a)S(b) < 0$, where $[a,b]$ is a suitable interval containing $c$.
 
Observe that, by using \cite{Corgnier}, Theorem 3.11, it is easy to prove 
that, if a given zero is of even order, then it is either  a local 
minimum or a local maximum for the series.

\end{remark}

\begin{corollary} \label{accumulation2}
Every zero of $S(X)$ whose order is odd is an accumulation point for the set $Z$.
\end{corollary}

\begin{proof}
If $c$ is a zero whose order is $2r+1$, there is an open neighbourhood and so, by possibly shrinking it, also a closed neighbourhood $ J = [c-\delta, c+\delta]$ where $S(x) < 0, \forall x < c, S(x) > 0, \forall x > c$ (or conversely). Therefore $c$ is the only zero of $S(X)$ in $J$. Now apply the above theorem.

\end{proof}

\begin{theorem} \label{accumulation-extremes}
Let $c$ be a zero of even order of $S(X)$ (so it is a local extreme). Then $c$ is an accumulation point of local extremes of the partial sums, $c$ and  the extremes having the same type.
\end{theorem}
\begin{proof}
We know that $c$ is a zero of odd order of the derivative $S'(X)$; as a consequence it is an accumulation point for the set of zeros of the partial sums $(S'(X))_n = S'_{n+1}(X)$, each of which is a local extreme of $S_n(X)$.
Now we observe that $S(X) = (X-c)^{2p}T(X)$, where $2p$ is the even order of the zero and $T(c) \ne 0$ (see \cite{Corgnier}, Theorem 3.11). Therefore (see \cite{Corgnier}, Theorem 3.7) $\frac{S(X)}{(X-c)^{2p}} = S^{(2p)}(c)+r(X-c)$, where $\lim_{X \to c}r(X-c) = 0$. It follows that the sign of $S(X)$ around $c$ is the same as the sign of $S^{(2p)}(c)$, which implies that $c$ is a maximum (minimum) iff $T(c) < 0 \ (> 0)$ or, which is the same, iff $S^{(2p)}(c) < 0 \ (> 0)$.  In order to see that a maximum is an accumulation point of maxima and a minimum of minima it is now enough to observe that, for $n$ large enough,  $(S^{(2p)}(c))_n = S^{(2p)}_{n-2p}(c)$ has the same sign as  $S^{(2p)}(c)$ since  $S^{(2p)}(c) = \lim_{n \to {\infty}}S_n^{(2p)}(c)$.

\end{proof}

The following example shows that a double zero of a series is approximated by extremes but not always by roots of the partial sums.  As usual $\e$ is nilpotent in $\KK = \QQ[\e]$.

\begin{example} \label{double zero} Let $F(X) = \sum_{n = 0}^{\infty}b_nX^n$ be any power series, $c$ any element of $\KK$, and set: $$T(X) = (x-c)^2F(X).$$

Then a straightforward computation on the partial sums $T_n(X), F_n(X)$ shows that $T_n(X) = (x-c)^2F_{n-1}(X)+x^n(c^2b_n-xb_{n-1})$.

Now choose $c = 1, F(X) = 2- \sum_{n = 1}^{\infty} \epsilon^{n+1}X^n = 2-\frac{\epsilon^2X}{1-\epsilon X}$. 
 Then both $F(X)$ and  $T(X)$ converge between $\frac{1}{2}$ and $2$ and moreover $F(x) > 1$ for every $x$ such that $1 \le x \le 2$, since it is the difference between $2$ and an infinitesimal element. Therefore $F_n(x) > \frac{1}{2}$ for $n$ large enough.  Hence we obtain the following inequality: $T_n(x) \ge \frac{(x-1)^2}{2}+x^n(-\e^{n+1}+x\e^n)$, where also the second term is strictly positive. Therefore $T_n(X)$ has no root converging to $1$.
\end{example}

\begin{remark} The set of zeros belonging to $A$ of the power series $S(X)$  is finite  (\ref{finitely many}). However there are power series whose domain is $\KK$ and having infinitely many zeros, as the following two examples show.
\end{remark}

\begin{example} \label {infinitely many zeros nilpotent} Let $\KK$ be a maximal ordered, complete field having a topologically nilpotent element $\epsilon \ $(the sets $ U_n = \{x \in \KK, |x| < \epsilon^n, n \in \NN\}$ form a basis for the neighbourhoods of $0$). Set: $$S(X) = \sum_{n = 0}^{\infty} (-1)^n \epsilon^{n^2}X^n.$$ Since $\lim(\epsilon^{n^2})^{\frac{1}{n}} = 0$, the domain $D_S$ of $S(X)$ is the whole $\KK$ (see \cite{Massaza2}, (IV), p. 137 and \cite{Corgnier}, Notation). We want to show that  

$S(\epsilon^m) < 0$ if $m = -4l-2, l \in \NN$

$S(\epsilon^m) > 0$ if $m = -4l, l \in \NN$.

To this purpose we observe that 

$S(\epsilon^{-4l}) = \sum(-1)^n \epsilon^{n^2-4ln}$, where $\epsilon^{n^2-4ln}$ is finite  for $n = 0$ and $n = 4l$, infinitesimal for $n > 4l$, infinitely large for $0 < n < 4l$. Since the maximum of $\phi(n) = 4ln-n^2$ is attained at $n = 2l$, the largest term of the series is $\e^{-4l^2}$ with positive sign, and it forces the series to attain a positive value.
 
An analogous argument shows that  $S(\epsilon^{-4l-2}) = \sum(-1)^n \epsilon^{n^2-4ln-2n}$ contains finitely many non-infinitesimal terms, among which  $-\epsilon^{-4l^2-4l-1}$ is infinitely large and most negative, so that it forces the series to attain a negative value.
 
As a consequence $S(X)$ vanishes infinitely many times, with one root at least between $\e^{-4l-2}$ and $\e^{-4l}$, thanks to the intermediate value theorem.
\end{example}

\begin{example} \label{infinitely many zeros nnilpotent} Let $\KK$ be a maximal ordered, complete field having no topologically nilpotent element, but a decreasing sequence of infinitesimal elements $(1 = \e_0 > \epsilon_1 > \epsilon_2 > \cdot \cdot \cdot)$ such that  the sets $(U_n = \{x \in \KK, |x| < \epsilon_n, n \in \NN\}$ form a basis for the neighbourhoods of $0$. As usual we assume that $\epsilon^i_n > \epsilon_{n+1}, \forall n \in \NN, \forall i \in \NN$. As a consequence the following holds: $\e_i^n > m\e_{i+1}, \forall,n,i,m \in \NN$.

 Set: $$S(X) = \sum_{n = 0}^{\infty} (-1)^n \epsilon_nX^n.$$ The domain $D_S$ of $S(X)$ contains at least $X = 1$, so that it is the whole $\KK$ (see \cite{Massaza2}, (IV), p. 137 and \cite{Corgnier}, Notation). 
 
 Let us compute $S(\epsilon_h^{-1})$:

 $S(\epsilon_h^{-1}) = 1-\e_1\e_h^{-1}+ \e_2\e_h^{-2} - \e_3\e_h^{-3} + ...+ (-1)^h \e_h\e_h^{-h} + (-1)^{h+1}  \e_{h+1}\e_h^{-(h+1)  } +(-1)^{h+2}  \e_{h+2}\e_h^{-(h+2) } \cdot \cdot \cdot$.
 
 We observe that 
 
 i) $  \e_{h+1}\e_h^{-(h+2)}$ is infinitesimal with respect to $\QQ$, since
 $  \e_{h+1}\e_h^{-(h+2)} < 1/n, \forall n$ is equivalent to  \   $ n \e_{h+1}  <\e_h^{h+2}$.

ii)  $  \e_{h+1}\e_h^{-(h+1)} >    \e_{h+2}\e_h^{-(h+2)}$, since this is equivalent to  $  \e_{h+1}\e_h > \e_{h+2}$; but we know that  $\e_{h+1}>\e_{h+2}$ and $\e_h>1$,  so we obtain the inequality.

As a consequence (if $u \in \NN, u \ge 1$) we obtain $|(-1)^{h+1} \e_{h+1}\e_h^{-(h+1)} +(-1)^{h+2}\e_{h+2}\e_h^{-(h+2)}+...+ (-1)^{h+u} \e_{h+u}\e_h^{-(h+u)} | < u \e_{h+1}\e_h^{-(h+1)} $. 

From the inequality $u<\e_h^{-1}, \forall u \in N$ we obtain $R_{h+1} < \e_{h+1} \e_h^{-(h+2) }  $, where $R_{h+1}$ is the remainder of order  $(h+1)$ of the series. Therefore $R_{h+1}$ is infinitesimal.

Let us now consider the first $h+1$ terms of the series, where $h \ge 2$:

$S_h(\e_h^{-1}) = 1- \e_1\e_h^{-1}+ \e_2\e_h^{-2} + ...+ (-1)^h  \e_h\e_h^{-h}$.

All terms, except the first one, are infinitely large and it holds: $m \e_i\e_h^{-i} < \e_h ^{1-h}, i<h, \forall m\in N$, since this is equivalent to $m \e_i < \e_h^{h+i-1}$. It follows that $|S_{h-1}(\e_h^{-1})|  <\e_h \e_h^{-h}$. Hence we see that the sign of the series concides with the sign of $(-1)^h  \e_h\e_h^{-h}$. Therefore we have:

$S(\e_h^{-1}) > 0$,  if $h$ is even \hskip 2cm  $S(\e_h^{-1}) < 0$, if $h$ is odd ($> 1$).

As a consequence $S(X)$ vanishes infinitely many times thanks to the intermediate value theorem.
\end{example}

\bigskip
\bigskip

Authors' address:

Dipartimento di Scienze Matematiche 

Politecnico di Torino 

Corso Duca degli Abruzzi 24 

10129 Torino Italy
\bigskip

Aknowledgements

The paper was written while C. Massaza and P. Valabrega were members of Indam-Gnsaga and P. Valabrega  was supported by the framework of PRIN 2010/11 \lq Geometria delle variet{\accent"12 a} algebriche\rq, cofinanced by MIUR.

\end{document}